\documentclass{article}

\usepackage{mathptmx}       
\usepackage{helvet}         
\usepackage{courier}        
\usepackage{type1cm}        
\usepackage{makeidx}         
\usepackage{graphicx}        
\usepackage{multicol}        
\usepackage[bottom]{footmisc}
\usepackage{amsmath,amsthm,amssymb,amsfonts}

\newtheorem{theorem}{Theorem}[subsection]
\newtheorem{lemma}[theorem]{Lemma}
\newtheorem{remark}{Remark}
\newtheorem{corollary}[theorem]{Corollary}

\makeindex             
\begin{document}

\title{Kinetic ES-BGK models for a multi-component gas mixture}
\author{Christian Klingenberg, Marlies Pirner and Gabriella Puppo}
%
%
\date{}
\maketitle


\abstract{We consider a multi component mixture of inert gas in the kinetic regime by assuming that the total number of particles of each species remains constant. In this article we shall illustrate our model for the case of two species. To account for thermal effects, we extend a BGK model based on the presence of a collision term for each possible interaction \cite{Pirner} by including ES-BGK effects. We prove consistency of the extended model like conservation properties, positivity of all temperatures, H-theorem and convergence to a global equilibrium in the shape of a global Maxwell distribution. 
}

\section*{Introduction}
 In this paper we shall concern ourselves with a kinetic description of gases. This is traditionally done via the Boltzmann equation for the density distributions $f_1$ and $f_2$. Under certain assumptions the complicated interaction terms of the Boltzmann equation can be simplified by a so called BGK approximation, consisting of a collision frequency multiplied by the deviation of the distributions from local Maxwellians. This approximation should be constructed in a way such that it  has the same main properties of the Boltzmann equation namely conservation of mass, momentum and energy, further it should have an H-theorem with its entropy inequality and the equilibrium must still be Maxwellian.  BGK  models give rise to efficient numerical computations, which are asymptotic preserving, that is they remain efficient even approaching the hydrodynamic regime \cite{Puppo_2007, Jin_2010,Dimarco_2014, Bennoune_2008,  Bernard_2015, Crestetto_2012}. However, the drawback of the BGK approximation  is its incapability of reproducing the correct Boltzmann hydrodynamic regime in the asymptotic continuum limit. Therefore, a modified version called ES-BGK approximation was suggested  by Holway in the case of one species \cite{Holway}. The H-Theorem of this model then was proven in \cite{Perthame} and existence and uniqueness of solutions in \cite{Yun}.
 
Here we shall focus on gas mixtures modelled via an ES-BGK approach. In the literature there is a BGK model for gas mixtures  suggested by Andries, Aoki and Perthame in \cite{AndriesAokiPerthame2002} which contains only one collision term on the right-hand side. Extensions of this model to an ES-BGK model for gas mixtures  are given by Groppi in \cite{ Groppi} or the model by Brull \cite{Brull} with an extension leading to a correct Prandtl number in the Navier Stokes equation, adapting the ES-BGK model for mixtures. 

In this paper we are interested in an extension to an ES-BGK model of a BGK model for gas mixtures \cite{Pirner} which just like the Boltzmann equation for gas mixtures contains a sum of collision terms on the right-hand side. Other examples of ES-BGK models for gas mixtures are the models of Gross and Krook \cite{gross_krook1956}, Hamel \cite{hamel1965},  Asinari \cite{asinari}. The advantage of this extended model is that we have free parameters to possibly being able to determine macroscopic physical constants like viscosity or heat conductivity when taking the limit to the Navier-Stokes equations.

The outline of the paper is as follows: in section 1.1 we will present the BGK model for two species  developed in \cite{Pirner}.  In section 1.2, we suggest extensions to an ES-BGK model for mixtures and prove the corresponding H-Theorem.

\section{The BGK approximation}
\label{sec:1}
In this section we will present the BGK model for a mixture of two species and mention its fundamental properties like the conservation properties and the H-theorem.

For simplicity in the following we consider a mixture composed of two different species, but the discussion can be generalized to multi species mixtures. Thus, our kinetic model has two distribution functions $f_1(x,v,t)> 0$ and $f_2(x,v,t) > 0$ where $x\in \Lambda \subset \mathbb{R}^3$ and $v\in \mathbb{R}^3$ are the phase space variables and $t\geq 0$ the time.  The distribution functions are determined by two equations to describe their time evolution. Furthermore we only consider binary interactions. So the particles of one species can interact with either themselves or with particles of the other species. In the model this is accounted for introducing two interaction terms in both equations. These considerations allow us to write formally the system of equations for the evolution of the mixture. The following structure containing a sum of the collision operator is also given in \cite{Cercignani, Cercignani_1975}.

 Furthermore, for any $f_1,f_2: \Lambda \subset \mathbb{R}^3 \times \mathbb{R}^3 \times \mathbb{R}^+_0 \rightarrow \mathbb{R}$ with $(1+|v|^2)f_1,(1+|v|^2)f_2 \in L^1(\mathbb{R}^3), f_1,f_2 \geq 0$ we relate the distribution functions to  macroscopic quantities by mean-values of $f_k$, $k=1,2$
\begin{align}
\int f_k(v) \begin{pmatrix}
1 \\ v  \\ m_k |v-u_k|^2 \\ m_k (v-u_k(x,t)) \otimes (v-u_k(x,t)) \end{pmatrix} 
dv =: \begin{pmatrix}
n_k \\ n_k u_k \\  3 n_k T_k \\ \mathbb{P}_k
\end{pmatrix} , \quad k=1,2,
\label{moments}
\end{align} 
where $n_k$ is the number density, $u_k$ the mean velocity and $T_k$ the mean temperature of species $k$, $k=1,2$. Note that in this paper we shall write $T_k$ instead of $k_B T_k$, where $k_B$ is Boltzmann's constant.

We are interested in a BGK approximation of the interaction terms. This leads us to define equilibrium distributions not only for each species itself but also for the two interspecies equilibrium distributions. We choose the collision terms   as BGK operators and denote them for future references by $Q_{11}, Q_{12}, Q_{21}$ and $Q_{22}$. Then the model can be written as:

\begin{align} \begin{split} \label{BGK}
\partial_t f_1 + \nabla_x \cdot (v f_1)   &= \nu_{11} n_1 (M_1 - f_1) + \nu_{12} n_2 (M_{12}- f_1),
\\ 
\partial_t f_2 + \nabla_x \cdot (v f_2) &=\nu_{22} n_2 (M_2 - f_2) + \nu_{21} n_1 (M_{21}- f_2), 
\end{split}
\end{align}
with the Maxwell distributions
\begin{align} 
\begin{split}
M_k(x,v,t) = \frac{n_k}{\sqrt{2 \pi \frac{T_k}{m_k}}^3 }  \exp({- \frac{|v-u_k|^2}{2 \frac{T_k}{m_k}}}),
\quad k=1,2,
\\
M_{kj}(x,v,t) = \frac{n_{kj}}{\sqrt{2 \pi \frac{T_{kj}}{m_k}}^3 }  \exp({- \frac{|v-u_{kj}|^2}{2 \frac{T_{kj}}{m_k}}}), \quad k,j=1,2,~ k \neq j,
\end{split}
\label{BGKmix}
\end{align}
where $\nu_{11} n_1$ and $\nu_{22} n_2$ are the collision frequencies of the particles of each species with itself, while $\nu_{12}$ and $\nu_{21}$ are related to interspecies collisions.
To be flexible in choosing the relationship between the collision frequencies, we now assume the relationship
\begin{align} 
\nu_{12}&=\varepsilon \nu_{21}, \hspace{3.1cm} 0 < \varepsilon \leq 1,
 \\
\nu_{11} &= \beta_1 \nu_{12}, \quad \nu_{22} = \beta_2 \nu_{21}, \quad \quad \beta_1, \beta_2 >0.
\label{coll}
\end{align}
The restriction $\varepsilon \leq 1$ is without loss of generality. If $\varepsilon >1$, exchange the notation $1$ and $2$ and choose $\frac{1}{\varepsilon}.$ In addition, we assume that all collision frequencies are positive.

The structure of the collision terms ensures that if one collision frequency $\nu_{kl} \rightarrow \infty$, the corresponding distribution function becomes a Maxwell distribution. In addition at global equilibrium, the distribution functions become Maxwell distributions with the same velocity and temperature (see section 2.8 in \cite{Pirner}).
The Maxwell distributions $M_1$ and $M_2$ in \eqref{BGKmix} have the same moments as $f_1$ and $f_2$, respectively. With this choice, we guarantee the conservation of mass, momentum and energy in interactions of one species with itself (see section 2.2 in \cite{Pirner}).
The remaining parameters $n_{12}, n_{21}, u_{12}, u_{21}, T_{12}$ and $T_{21}$ will be determined using conservation of total momentum and energy, together with some symmetry considerations.

%
%
%
%
%
%
If we assume that \begin{align} n_{12}=n_1 \quad \text{and} \quad n_{21}=n_2,  
\label{density} 
\end{align}
 \begin{align}
u_{12}= \delta u_1 + (1- \delta) u_2, \quad \delta \in \mathbb{R},
\label{convexvel}
\end{align} 
and
\begin{align}
\begin{split}
T_{12} &=  \alpha T_1 + ( 1 - \alpha) T_2 + \gamma |u_1 - u_2 | ^2,  \quad 0 \leq \alpha \leq 1, \gamma \geq 0 ,
\label{contemp}
\end{split}
\end{align}
we have conservation of the number of particles, of total momentum and total energy provided that
\begin{align}
u_{21}=u_2 - \frac{m_1}{m_2} \varepsilon (1- \delta ) (u_2 - u_1),
\label{veloc}
\end{align}
and
\begin{align}
\begin{split}
T_{21} =\left[ \frac{1}{3} \varepsilon m_1 (1- \delta) \left( \frac{m_1}{m_2} \varepsilon ( \delta - 1) + \delta +1 \right) - \varepsilon \gamma \right] |u_1 - u_2|^2 \\+ \varepsilon ( 1 - \alpha ) T_1 + ( 1- \varepsilon ( 1 - \alpha)) T_2,
\label{temp}
\end{split}
\end{align}
 see Theorem 2.1, Theorem 2.2 and Theorem 2.3 in \cite{Pirner}.

 We see that without using an ES-BGK extension, we already have three free parameters in \eqref{convexvel} and \eqref{contemp} in order to match coefficients like the Fick's constant or the heat conductivity in the Navier-Stokes equations.  But when we derive the Navier-Stokes equations by a Chapman-Enskog expansion $f_k = f_k^0 + \tilde{\epsilon} f_k^1 + \tilde{\epsilon}^2 f_k^2 + \cdots$, one can show that $|u_1 - u_2|^2$ is of order $\tilde{\epsilon}^2$, so $\gamma$ from \eqref{contemp} does not appear in the first order Navier-Stokes equations and therefore cannot be used to match parameters there. \\
In order to ensure the positivity of all temperatures, we need to impose restrictions on $\delta$ and $\gamma$, 
 \begin{align}
0 \leq \gamma  \leq \frac{m_1}{3} (1-\delta) \left[(1 + \frac{m_1}{m_2} \varepsilon ) \delta + 1 - \frac{m_1}{m_2} \varepsilon \right],
 \label{gamma}
 \end{align}
and
\begin{align}
 \frac{ \frac{m_1}{m_2}\varepsilon - 1}{1+\frac{m_1}{m_2}\varepsilon} \leq  \delta \leq 1,
\label{gammapos}
\end{align}
see Theorem 2.5 in \cite{Pirner}.\\ \\ This summarizes our kinetic model \eqref{BGK} in of two species that contains three free parameters. More details can be found in \cite{Pirner}.

\section{Extensions to an ES-BGK approximation}
\label{sec:2}
\subsection{Extension of the single relaxation terms}
\label{sec:2.1}

Motivated by the need to find a two species kinetic model that allows us to model physical
parameters better we extend the above model by generalizing the Maxwellians. The simplest choice is to only  replace the collision operators which represent the collisions of a species with itself by the ES-BGK collision operator for one species suggested in \cite{AndriesPerthame2001}. Then the model can be written as:

\begin{align} \begin{split} \label{ESBGKsimple}
\partial_t f_k + \nabla_x \cdot (v f_k)   &= \nu_{kk} n_k (G_k - f_k ) + \nu_{kj} n_j ( M_{kj} - f_k), \quad k,j=1,2, ~ j \neq k,
\end{split}
\end{align}
with the modified Maxwell distributions
\begin{align} 
\begin{split}
G_k(x,v,t)&= \frac{n_k}{\sqrt{det(2 \pi \frac{\mathcal{T}_k}{m_k})}} \exp(- \frac{1}{2} (v-u_k) \cdot (\frac{\mathcal{T}_k}{m_k})^{-1} \cdot (v-u_k)), \quad k=1,2,
\end{split}
\label{ESBGKmixsimple}
\end{align}
and $M_{12}, M_{21}$ the Maxwellians described in the previous section.
 $G_1$ and $G_2$ have the same densities, velocities and pressure tensors as $f_1$ respective $f_2$, so we still guarantee the conservation of mass, momentum and energy in interactions of one species with itself. 
Since the first term describes the interactions of a species with itself, it should correspond to the single ES-BGK collision operator suggested in \cite{AndriesPerthame2001}. So we choose $\mathcal{T}_1$ and $\mathcal{T}_2$ as 
\begin{align}
\mathcal{T}_k= (1- \mu_k) T_k \textbf{1} + \mu_k \frac{\mathbb{P}_k}{n_k}, 
\label{ten}
\end{align}
 with $\mu_k \in \mathbb{R}$, $k=1,2$ being free parameters which we can choose in a way to fix physical parameters in the Navier-Stokes equations. So, all in all, together with the parameters in the mixture Maxwellians \eqref{convexvel} and \eqref{contemp} we now have five free parameters. 

Since we wrote $\mathcal{T}_k^{-1}$ we have to check if $\mathcal{T}_k$ is invertible. Otherwise the model is not well-posed. For the one species tensor this is done by the following Theorem proven in \cite{AndriesPerthame2001}.
\begin{theorem}
Assume that $f_k>0$. Then $\frac{{\mathbb{P}_k}}{n_k }$  has strictly positive eigenvalues. If we further assume that $ - \frac{1}{2} \leq \mu_k \leq 1,$ then $\mathcal{T}_k$ has strictly positive eigenvalues and therefore $\mathcal{T}_k$ is invertible.
\end{theorem}
\subsubsection{Equilibrium and entropy inequality}
\label{subsubsec:2}
In global equilibrium when $f_1$ and $f_2$ are independent of $x$ and $t$, the right- hand side of \eqref{ESBGKsimple} has to be zero. In this case we get 
$$ f_1 = \frac{1}{\nu_{11} n_1 + \nu_{12} n_2}( \nu_{11} n_1 G_1 + \nu_{12} n_2 M_{12}).$$
If we compute the velocities of this expression, we can deduce $u_1=u_2$ for $\delta \neq 1$.
If we compute the temperatures of this expression using $u_1=u_2$, we get
\begin{align*}
 T_1 = \frac{1}{\nu_{11} n_1 + \nu_{12} n_2} ( \nu_{11} n_1 T_1 + \nu_{12} n_2 ( \alpha T_1 + ( 1- \alpha) T_2)), \\
\end{align*}
which is  equivalent to $T_1=T_2$ for $\alpha \neq 1$. So let $T:= T_1= T_2$ and use $u_1=u_2$.
If we compute pressure tensors, we get
\begin{align*}
(\nu_{11} n_1 + \nu_{12} n_2) \mathbb{P}_1 &= \nu_{11} n_1  \mathcal{T}_1 + \nu_{12} n_2 T_{12}\\ &= \nu_{11} n_1 ( 1- \mu_1) T \textbf{1} + \nu_{11} n_1 \mu_1 \mathbb{P}_1 + \nu_{12} n_2 T \textbf{1},
\end{align*}
which is equivalent to 
\begin{align*}
(\nu_{11} n_1 + \nu_{12} n_2 -\nu_{11} n_1 \mu_1 ) \mathbb{P}_1 = (\nu_{11} n_1 + \nu_{12} n_2 -\nu_{11} n_1 \mu_1) T \textbf{1},
\end{align*}
which is $\mathbb{P}_1 = T \textbf{1}$ for $\delta, \alpha \neq 1$, $\mu_1 \leq 1$.
This means that the pressure tensor of $f_1$ and $f_2$ is diagonal and $f_1,f_2$ are Maxwellian distributions with equal mean velocity and temperature. $\delta=1$ or $\alpha =1$ are cases in which the mixture Maxwellians do not contain the velocity or the temperature of the other species, see \eqref{convexvel} and \eqref{contemp}. In this case the two gases do not exchange information and a global equilibrium cannot be reached.                                                                                                                                                                                                                                                                                                                                                                                                                                                                                                                                                                                                                                                                                                                                                                                                                                                                                                                                                                                                                                                                                                                                                                                                                                                                                                                                                                                                                                                                                                                                                                                                                                                                                                                                                                                                                                                                                                                                                                                                                                                                                                                                                                                                                                                                                                                                                                                                                                                                                                                                                                                                                                                                                                                                                                                                                                                                                                                                                                                                                                                                                                                                                                                                                                                         
\begin{theorem}[H-theorem for the mixture]
Assume that $f_1, f_2 >0$ are solutions to \eqref{BGK}.
Assume the relationship between the collision frequencies \eqref{coll} , the conditions for the interspecies Maxwellians  \eqref{convexvel}, \eqref{veloc}, \eqref{contemp} and \eqref{temp} and the positivity of the temperatures \eqref{gamma}, then
{\small
$$
\int (\ln f_1) ~ Q_{11}(f_1,f_1) + (\ln f_1) ~ Q_{12}(f_1,f_2) dv \\+ \int (\ln f_2) ~ Q_{22}(f_2, f_2)+ (\ln f_2) ~ Q_{21}(f_2, f_1) dv\leq 0,
$$}
 with equality if and only if $f_1$ and $f_2$ are Maxwell distributions with equal velocity and temperature. 
\label{H-theoremsimple}
\end{theorem}
\begin{proof}
The fact that $\int \ln f_k Q_{kk}(f_k,f_k) dv \leq 0, \quad k=1,2$ with a criteria for equality follows from the H-Theorem of the ES-BGK model for one species, see \cite{AndriesPerthame2001}. The fact that $\int \ln f_1 Q_{12} (f_1,f_2) dv + \int \ln f_2 Q_{21} ( f_1,f_2) dv \leq 0$ with a corresponding criteria for equality follows from the H-Theorem of the BGK model for mixtures, see Theorem 2.7 in \cite{Pirner}.
\end{proof}

\subsection{Alternative extensions to an ES-BGK model}
In this subsection we also want to replace the scalar temperatures in the mixture Maxwellians by a tensor.
In the first model the terms $(v_j - u_{kj})f_k (v_i - u_{ki})$ for $i\neq j$ do not appear in the relaxation operator. To obtain a more detailed description of the viscous effects in the mixture we take into account these cross terms during the relaxation process. Then the model can be written as:

\begin{align} \begin{split} \label{ESBGK}
\partial_t f_k + \nabla_x \cdot (v f_k)   &= \nu_{kk} n_k (G_k - f_k ) + \nu_{kj} n_j ( G_{kj} - f_k), \quad k=1,2, k \neq j, 
\end{split}
\end{align}
with the modified Maxwell distributions
\begin{align} 
\begin{split}
G_k(x,v,t)= \frac{n_k}{\sqrt{\det(2 \pi \frac{\mathcal{T}_k}{m_k})}} \exp(- \frac{1}{2} (v-u_k) \cdot (\frac{\mathcal{T}_k}{m_k})^{-1} \cdot (v-u_k)) \quad k=1,2,
\\
G_{kj}(x,v,t) = \frac{n_k}{\sqrt{\det(2 \pi \frac{\mathcal{T}_{kj}}{m_k})}} \exp(- \frac{1}{2} (v-u_{kj}) \cdot (\frac{\mathcal{T}_{kj}}{m_k})^{-1} \cdot (v-u_{kj})) \quad k=1,2, k \neq j.
\end{split}
\label{ESBGKmix}
\end{align}
Again, the conservation of mass, momentum and energy in interactions of one species with itself is ensured by this choice of the modified Maxwell distributions $G_1$ and $G_2$ which have the same densities, velocities and pressure tensor as $f_1$ and $f_2$, respectively. In addition, the choice of the densities in $G_{12}$ and $G_{21}$, we also guarantee conservation of mass in interactions of one species with the other one.
If we extend $T_{12}$ and $T_{21}$ in the same fashion to a tensor as in the case of one species, we obtain
{\small
\begin{align}
\mathcal{T}_{12} &=(1- \mu_{12}) (\alpha T_1 + (1- \alpha) T_2 ) \textbf{1} + \mu_{12} \frac{\alpha \mathbb{P}_1 + (1- \alpha) \mathbb{P}_2 }{n_1}+ \gamma |u_1 - u_2|^2  \textbf{1}, \label{tau12a}
 \\
 \begin{split}
\mathcal{T}_{21} &= (1- \mu_{21} ) ((1- \varepsilon (1- \alpha)) T_2 + \varepsilon (1- \alpha) T_1)  \textbf{1} \\&+ \mu_{21} \frac{(1- \varepsilon (1- \alpha)) \mathbb{P}_2 + \varepsilon (1- \alpha) \mathbb{P}_1 }{n_2}+ (\frac{1}{3} \varepsilon m_1 (1- \delta)( \frac{m_1}{m_2} \varepsilon ( \delta - 1) + \delta + 1) - \varepsilon \gamma)|u_1 - u_2|^2 \textbf{1}. \label{tau21a}
\end{split}
\end{align}}
If we check the equilibrium distributions as in section 1.2.1.1, we obtain the following restrictions on $\mu_{12}$ and $\mu_{21}$ given by
\begin{align}
\mu_{12} = 1+ (1- \mu_1) \frac{n_1}{n_2} \frac{ \nu_{11}}{\nu_{12}},
\label{res12}
\end{align}
and
\begin{align}
\begin{split}
 \frac{1}{n_1^2} [ -(\alpha -1)^2 \mu_{12}^2 n_2^2 \nu_{12}^2 + \frac{n_1}{n_2^2}(( \frac{\mu_{21}}{\varepsilon} - \mu_{21} + \alpha \mu_{21}) n_1 \nu_{12} + (\mu_2 -1) n_2 \nu_{22} )\\ \cdot ( n_1 ((\alpha -1) \mu_{21} n_1 + \frac{1}{\varepsilon} ( \mu_{21} -1) n_2 ) \nu_{12} + (\mu_2 -1) n_2^2 \nu_{22} )]=0,
 \end{split}
 \label{res21}
\end{align}
An alternative choice to \eqref{tau12a},\eqref{tau21a}, which is less complicated, is given by
\begin{align}
\mathcal{T}_{12} &= \alpha \frac{\mathbb{P}_1}{n_1} + ( 1 - \alpha) T_2 \textbf{1} + \gamma |u_1 - u_2|^2 \textbf{1}, \label{tau12}
 \\
 \begin{split}
\mathcal{T}_{21} &= (1-\varepsilon (1-\alpha ) )\frac{\mathbb{P}_2}{n_2} + \varepsilon( 1 - \alpha) T_1 \textbf{1}\\& +  ( \frac{1}{3} \varepsilon m_1 (1 - \delta) ( \frac{m_1}{m_2} \varepsilon ( \delta - 1) + \delta + 1) - \varepsilon \gamma)|u_1 - u_2|^2 \textbf{1}. 
\end{split}
\label{tau21}
\end{align}
This choice still contains the temperature of gas $1$, since the trace of the pressure tensor is the temperature.

In \eqref{tau12} compared to \eqref{tau12a} we replace only the temperature $T_1$ of species $1$ by the pressure tensor $\mathbb{P}_1$ while we keep the temperature $T_2$. This asymmetric choice can be motivated by the theory of "persistence of velocity" described by Jeans in \cite{Jeans} and \cite{Jeans2}. 
He argues that in the post-collisional speed of particle $1$ there is a memory of the pre-collisional speed of particle $1$. In the single species BGK equation this yields to the choice of
$$\mathcal{T} = (1- \mu) T \textbf{1} + \mu \mathbb{P}, \quad -\frac{1}{2} \leq \mu \leq 1,$$
the tensor chosen in the well-known ES-BGK model, where $\mu \mathbb{P}$ preserves the memory of the off-equilibrium content of the pre-collisional velocity. This can be rewritten as
$$ \mathcal{T} = T \textbf{1} + \mu \text{traceless}[\mathbb{P}], $$
where $\text{traceless}[\mathbb{P}]$ denotes the traceless part of $\mathbb{P}$. So the off-equilibrium part is contained in $\mu \text{traceless}[\mathbb{P}].$ 
Doing this analogously for two species we arrive at 
$$\mathcal{T}_{12} = T_{12} \textbf{1} + \frac{\alpha}{n_1} \text{traceless}[\mathbb{P}_1].$$
If we plug in the definition of $T_{12}$ given by \eqref{contemp}, we end up with \eqref{tau12}.

With the second choice the model is well-defined, because $\mathcal{T}_{12}$ and $\mathcal{T}_{21}$ are invertible as a combination of strictly positive matrices as soon as all coefficients in front of these matrices are positive, which is the case due to \eqref{gamma} and \eqref{gammapos}.
The first choice needs additional conditions coming from the restrictions on $\mu_{12}$ and $\mu_{21}$ given by \eqref{res12} and \eqref{res21}. The first one leads to $$ \mu_1 \leq \frac{n_2}{n_1} \frac{\nu_{12}}{\nu_{11}} +1,$$
such that $\mu_{12}$ given by \eqref{res12} is positive. The requirement of positivity of $\mu_{21}$ leads to a corresponding restriction on $\mu_2$ using \eqref{res21}.
\subsubsection{Equilibrium and entropy inequality}
\label{sec:3}
The aim of this subsection is to discuss the property of equilibrium and the entropy inequality for the alternative extensions described in subsection 2.2 with the tensors \eqref{tau12a}, \eqref{tau21a} respective \eqref{tau12}, \eqref{tau21}. For the tensors \eqref{tau12a}, \eqref{tau21a} we proved the property of equilibrium and the H-Theorem in subsection 2.1.1 in the particular case for $\mu_{12}= \mu_{21}=0$ for simplicity, but we can also prove it in the general case.
In this section we will prove an entropy inequality for the alternative model \eqref{tau12},\eqref{tau21}. First we will check that the equilibrium distributions are Maxwellians.
In global equilibrium when $f_1$ and $f_2$ are independent of $x$ and $t$, the right- hand side of \eqref{ESBGK} has to be zero. In this case we get 
$$ f_1 = \frac{1}{1 + \frac{1}{\beta_1^2}\frac{n_2}{n_1}}( G_1 + \frac{1}{\beta_1^2} \frac{n_2}{n_1} G_{12}).$$
If we compute the temperatures of this expression, we get
\begin{align*}
 T_1 = \frac{1}{1 + \frac{1}{\beta_1^2}\frac{n_2}{n_1}} ( T_1 + \frac{1}{\beta_1^2} \frac{n_2}{n_1} ( \alpha T_1 + ( 1- \alpha) T_2)), \\
\end{align*}
which is  equivalent to $T_1=T_2$ for $\alpha \neq 1$. So denote $T:= T_1= T_2$. 
If we compute pressure tensors, we get
\begin{align*}
(1 + \frac{1}{\beta_1^2}\frac{n_2}{n_1}) \mathbb{P}_1 &=  \mathcal{T}_1 + \frac{1}{\beta_1^2}\frac{n_2}{n_1} \mathcal{T}_{12}\\ &=  ( 1- \nu_1) T +  \nu_1 \mathbb{P}_1 + \frac{1}{\beta_1^2}\frac{n_2}{n_1} \alpha \mathbb{P}_1 +\frac{1}{\beta_1^2}\frac{n_2}{n_1} (1- \alpha)T \textbf{1}
\end{align*}
which is equivalent to 
\begin{align*}
((1-\nu_1) + \frac{1}{\beta_1^2}\frac{n_2}{n_1} (1- \alpha)) \mathbb{P}_1 = ((1-\nu_1) + \frac{1}{\beta_1^2}\frac{n_2}{n_1} (1- \alpha)) T \textbf{1},
\end{align*}
which is $\mathbb{P}_1 = T \textbf{1}$ for $\nu_1,\alpha \neq 1$. That means that the pressure tensors of $f_1$ and $f_2$ are diagonal and they are Maxwellian distributions with equal mean velocity and temperature.                                                                                                                                                                                                                                                                                                                                                                                                                                                                                                                                                                                                                                                                                                                                                                                                                                                                                                                                                                                                                                                                                                                                                                                                                                                                                                                                                                                                                                                                                                                                                                                                                                                                                                                                                                                                                                                                                                                                                                                                                                                                                                                                                                                                                                                                                                                                                                                                                                                                                                                                                                                                                                                                                                                                                                                                                                                                                                                                                                                                                                                                                                                                                                                                                                                            

Next, we want to prove the H-Theorem of the simpler model \eqref{tau12} and \eqref{tau21}. For this proof, we need the following lemmas.
\begin{lemma}[Brunn-Minkowski inequality]
Let $0 \leq a \leq 1$ and $A,B$ positive symmetric matrices, then 
$$ \det (aA + (1-a) B) \geq (\det A)^a (\det B)^{1-a}.$$
\end{lemma}
\begin{proof}
The proof is given in \cite{AndriesPerthame2001}.
\end{proof}
\begin{lemma}
Assuming \eqref{tau12} and \eqref{tau21} and the positivity of all temperatures and pressure tensors  \eqref{gamma}, we have the following inequality
\begin{align*}
S:= ( \det\mathcal{T}_{12})^{\varepsilon} ( \det \mathcal{T}_{21}) \geq ( \det \frac{\mathbb{P}_1}{n_1})^{\varepsilon} \det \frac{\mathbb{P}_2}{n_2}.
\end{align*}
\label{inequ}
\end{lemma}
\begin{proof}
Using the definition of $\mathcal{T}_{12}$ we get
\begin{align*}
\det \mathcal{T}_{12}= \det ( \alpha \frac{\mathbb{P}_1}{n_1} + (1- \alpha) T_2 \textbf{1} + \gamma |u_1-u_2|^2 \textbf{1} ).
\end{align*}
Since $\gamma$ is non-negative, we can estimate the expression by dropping the positive term on the diagonal $\gamma |u_1-u_2|^2 \textbf{1}$
\begin{align*}
\det \mathcal{T}_{12}\geq \det ( \alpha \frac{\mathbb{P}_1}{n_1} + (1- \alpha) T_2 \textbf{1}).
\end{align*}
With the Brunn-Minkowski-inequality we obtain
\begin{align*}
\det \mathcal{T}_{12}\geq ( \det \frac{\mathbb{P}_1}{n_1} )^{\alpha} ( \det T_2 \textbf{1} )^{ 1- \alpha}.
\end{align*}
In a similar way, we can show it for $\mathcal{T}_{21}$, so all in all we get
\begin{align*}
S \geq ( \det \frac{\mathbb{P}_1}{n_1} )^{\alpha \varepsilon} ( \det T_2 \textbf{1} )^{ \varepsilon(1- \alpha)} ( \det \frac{\mathbb{P}_2}{n_2} )^{1- \varepsilon(1-\alpha)} ( \det T_1 \textbf{1} )^{ \varepsilon(1- \alpha)}.
\end{align*}
Consider the logarithm of this equation
\begin{align*}
\ln S \geq \varepsilon \alpha \ln \left( \det \left( \frac{\mathbb{P}_1}{n_1} \right) \right) + \varepsilon (1- \alpha) \ln \left( \det \left( T_2 \textbf{1} \right) \right) \\ + (1- \varepsilon ( 1- \alpha)) \ln \left( \det \left( \frac{\mathbb{P}_2}{n_2} \right) \right) + \varepsilon (1- \alpha ) \ln \left( \det \left(T_1 \textbf{1} \right) \right).
\end{align*}
We use that $\ln \left( \det \left(T_i \textbf{1} \right) \right)= \text{Tr} ( \ln \left(T_i \textbf{1} \right) )$, $T_i= \text{Tr} \frac{\mathbb{P}_i}{3 n_i}$ and denote  the eigenvalues of $\frac{\mathbb{P}_i}{n_i}$ by $\lambda_{i,1}, \lambda_{i,2}$ and $ \lambda_{i,3}$. Since the pressure tensors are symmetric, we can diagonalize them and use that $T_i= \text{Tr} \frac{\mathbb{P}}{3 n_i} = \lambda_{i,1} + \lambda_{i,2} + \lambda_{i,3}$.
\begin{align*}
\ln S \geq \varepsilon \alpha ( \ln \lambda_{1,1} + \ln \lambda_{1,2}+ \ln \lambda_{1,3}) + \varepsilon ( 1- \alpha) 3 \ln \frac{1}{3}(\lambda_{1,1}+\lambda_{1,2}+\lambda_{1,3}) \\+(1- \varepsilon (1- \alpha)) ( \ln \lambda_{2,1} + \ln \lambda_{2,2}+ \ln \lambda_{2,3}) + \varepsilon ( 1- \alpha) 3 \ln \frac{1}{3}(\lambda_{2,1}+\lambda_{2,2}+\lambda_{2,3}).
\end{align*}
Since $ln$ is concave, we can estimate $\ln \frac{1}{3}(\lambda_{1,1}+\lambda_{1,2}+\lambda_{1,3})$ from below by \\$ \frac{1}{3} (\ln \lambda_{1,1} + \ln \lambda_{1,2} + \ln \lambda_{1,3})$ and obtain
\begin{align*}
\ln S \geq \varepsilon \ln \left( \det \left( \frac{\mathbb{P}_1}{n_1} \right) \right) + \varepsilon (1-\alpha) \ln \left( \det \left( \frac{\mathbb{P}_2}{n_2} \right) \right).
\end{align*}
This is equivalent to the required inequality.
\end{proof}
\begin{remark}
From the case of one species ES-BGK model we know that
$$ \int G_k \ln G_k dv \leq \int G_{k, \mu_k=1} \ln G_{k, \mu_k=1} dv \leq \int f_k \ln f_k dv,$$ for $k=1,2$, see \cite{AndriesPerthame2001}, where $G_{k, \mu_k=1}$ denotes the modified Maxwellian where $\mu_k=1$ in the  tensor \eqref{ten}.
\label{one}
\end{remark}
\begin{theorem}[H-theorem for mixture]
Assume $f_1, f_2 >0$.
Assume the relationship between the collision frequencies \eqref{coll}, the conditions for the interspecies Maxwellians \eqref{convexvel}, \eqref{veloc}, \eqref{tau12} and \eqref{tau21} and the positivity of the temperatures \eqref{gamma}, then
{\small
$$
\int (\ln f_1) ~ Q_{11}(f_1,f_1) + (\ln f_1) ~ Q_{12}(f_1,f_2) dv + \int (\ln f_2) ~ Q_{22}(f_2, f_2)+ (\ln f_2) ~ Q_{21}(f_2, f_1) dv\leq 0,
$$}
with equality if and only if $f_1$ and $f_2$ are Maxwell distributions with equal mean velocity and temperature. 
\label{H-theorem}
\end{theorem}
\begin{proof}
The fact that $\int \ln f_k Q_{kk}(f_k,f_k) dv \leq 0$, $k=1,2$ is shown in proofs of the H-theorem of the single ES-BGK-model, for example in \cite{AndriesPerthame2001}.
In both cases we have equality if and only if $f_1=M_1$ and $f_2 = M_2$. \\
Let us define 
$$
S(f_1,f_2) :=\nu_{12} n_2 \int \ln f_1  ( G_{12}- f_1 )dv + \nu_{21} n_1 \int \ln f_2  ( G_{21}- f_2)dv .
$$
The task is to prove that $S(f_1, f_2)\leq 0$.
Since the function $H(x)= x \ln x -x$ is strictly convex for $x>0$, we have $H'(f) (g-f) \leq H(g) - H(f)$ with equality if and only if $g=f$. So \begin{align}
(g-f) \ln f  \leq g \ln g - f \ln f +f -g .
\label{convex}
\end{align}
Consider now $S(f_1,f_2)$ and apply the inequality \eqref{convex} to each of the two terms in $S$.
$$
S (f_1,f_2)\leq
\nu_{12} n_2 \left[ \int G_{12} \ln G_{12}  dv - \int f_1 \ln f_1 dv - \int G_{12} dv + \int f_1 dv\right]$$ $$ + \nu_{21} n_1 \left[\int G_{21} \ln G_{21} dv - \int f_2 \ln f_2 dv - \int G_{21} dv + \int f_2 dv \right] ,
$$
with equality if and only if $f_1=G_{12}$ and $f_2=G_{21}$. If we compute the velocities of $f_1=G_{12}$ and $f_2=G_{21}$, we can deduce $u_1=u_{12}$ and $u_2=u_{21}$ which lead to $u_1=u_2$ using the definitions of $u_{12}, u_{21}$ given by \eqref{convexvel} and \eqref{veloc}. Analogously, computing the temperatures, we get $T_{12}=T_{21}=T_1=T_2=:T$. Finally, computing the pressure tensors, we obtain $\frac{\mathbb{P}_1}{n_1}= \frac{\mathbb{P}_2}{n_2}= T \mathbf{1}$, which means that we have equality if and only if $f_1$ and $f_2$ are Maxwellians with equal temperatures and velocities.  \\
Since $G_{12}$ and $f_1$ have the same density and $G_{21}$ and $f_2$ have the same density too,  the right-hand side reduces to
$$
\nu_{12} n_2 ( \int G_{12} \ln G_{12} dv - \int f_1 \ln f_1 dv )+ \nu_{21} n_1 (\int G_{21} \ln G_{21} dv - \int f_2 \ln f_2 dv ).
$$
Since $\int G \ln G dv = n \ln(\frac{n}{\sqrt{\det(\frac{2 \pi \mathcal{T}}{m})}})- \frac{3}{2} n$ for $G=\frac{n}{\sqrt{\det(\frac{ 2 \pi \mathcal{T}}{m})}^3}e^{-(v-u)\cdot (\frac{\mathcal{T}}{m})^{-1}\cdot (v-u)},$  we will have that
$$
\nu_{12} n_2  \int G_{12} \ln G_{12}  dv + \nu_{21} n_1 \int G_{21} \ln G_{21}  dv $$$$\leq \nu_{21} n_1 \int G_{2, \mu_2=1} \ln M_{2, \mu_2=1}  dv  + \nu_{12} n_2  \int G_{1, \mu_1=1} \ln G_{1, \mu_1=1}  dv, 
$$
provided that
\begin{align*}
\nu_{12} n_2 n_1 \ln \frac{n_1}{\sqrt{\det (2 \pi \frac{\mathcal{T}_{12}}{m_1})}} +\nu_{21} n_2 n_1 \ln \frac{n_2}{\sqrt{\det( 2 \pi \frac{\mathcal{T}_{21})}{m_2}}} \\ \leq \nu_{12} n_2 n_1 \ln \frac{n_1}{\sqrt{\det(2 \pi \frac{\mathbb{P}_1}{m_1})}} +\nu_{21} n_2 n_1 \ln \frac{n_2}{\sqrt{\det(2 \pi \frac{\mathbb{P}_2}{m_2})}},
\end{align*}
which is equivalent to the condition 
$$
( \det\mathcal{T}_{12})^{\varepsilon} ( \det \mathcal{T}_{21}) \geq ( \det \frac{\mathbb{P}_1}{n_1})^{\varepsilon} \det \frac{\mathbb{P}_2}{n_2},
$$
proven in Lemma \ref{inequ}.
\\
With this inequality we get
\begin{align*}
S(f_1,f_2) \leq
&\nu_{12} n_2 [ \int G_{1, \mu_1=1} \ln G_{1, \mu_1=1}  dv - \int f_1 \ln f_1 dv ]\\&+ \nu_{21} n_1 [ G_{2, \mu_2=1} \ln G_{2, \mu_2=1} dv - \int f_2 \ln f_2 dv ] \leq 0 .
\end{align*}
The last inequality follows from remark \ref{one}. Here we also have equality if and only if $f_1=M_1$ and $f_2=M_2$, but since we already noticed that equality also implies $f_1=G_{12}$ and $f_2=G_{21}$.
\end{proof}
Define the total entropy $H(f_1,f_2) = \int (f_1 \ln f_1 + f_2 \ln f_2) dv$. We can compute 
$$ \partial_t H(f_1,f_2) + \nabla_x \cdot \int ( f_1 \ln f_1 + f_2 \ln f_2 ) v dv  = S(f_1,f_2),$$ by multiplying the BGK equation for the species $1$ by $\ln f_1$, the BGK equation for the species $2$ by $\ln f_2$ and integrating the sum with respect to $v$.
 \begin{corollary}[Entropy inequality for mixtures]
Assume $f_1, f_2 >0$. Assume  a fast enough decay of $f$ to zero for $v\rightarrow \infty$.
Assume relationship \eqref{coll}, the conditions \eqref{convexvel}, \eqref{veloc}, \eqref{tau12} and \eqref{tau21} and the positivity of the temperatures \eqref{gamma} , then we have the following entropy inequality
$$
\partial_t \left(\int   f_1 \ln f_1  dv + \int f_2 \ln f_2 dv \right) + \nabla_x \cdot \left(\int  v f_1 \ln f_1  dv + \int v f_2 \ln f_2 dv \right) \leq 0,
$$
with equality if and only if $f_1$ and $f_2$ are Maxwell distributions with equal bulk velocity and temperature. 
\end{corollary}

In summary the ES-BGK models \eqref{ESBGKsimple}, \eqref{ESBGK} have five free parameters. We expect this will aid in determining macroscopic physical constants, analogously to how it is done in \cite{Groppi}.

\end{document}